\documentclass{amsart}

\usepackage{amsmath}
\usepackage{amssymb}
\usepackage{hyperref}
\usepackage{enumitem,xcolor}
\numberwithin{equation}{section}

\newtheorem{theorem}{Theorem}[section]
\newtheorem{lemma}[theorem]{Lemma}
\newtheorem{corollary}[theorem]{Corollary}
\newtheorem{proposition}[theorem]{Proposition}

\theoremstyle{definition}
\newtheorem{definition}[theorem]{Definition}

\theoremstyle{remark}

\newcommand\R{\mathbb{R}}

\newcommand\C{\mathbb{C}}
\newcommand\N{\mathbb{N}}

\newcommand\real{\mathrm{Re}\,}

\newcommand{\qtq}[1]{\quad\text{#1}\quad}

\renewcommand{\epsilon}{\varepsilon}

\let\Im=\undefined\DeclareMathOperator{\Im}{Im}

\allowdisplaybreaks

\begin{document}

\title{Turbulent Threshold for Continuum Calogero-Moser Models}

\author[J.~Hogan]{James Hogan}
\address{Department of Mathematics, University of California, Los Angeles, CA 90095, USA}
\email{jameshogan@math.ucla.edu}

\author[M.~Kowalski]{Matthew Kowalski}
\address{Department of Mathematics, University of California, Los Angeles, CA 90095, USA}
\email{mattkowalski@math.ucla.edu}

\maketitle

\begin{abstract}
    We determine the sharp mass threshold for Sobolev norm growth for the focusing continuum Calogero--Moser model.
    It is known that below the mass of $2\pi$, solutions to this completely integrable model enjoy uniform-in-time $H^s$ bounds for all $s \geq 0$.
    In contrast, we show that for arbitrarily small $\epsilon > 0$ there exists initial data $u_0 \in H^\infty_+$ of mass $2\pi + \epsilon$ such that the corresponding maximal lifespan solution $u : (T_-, T_+) \times \R \to \C$ satisfies $\lim_{t \to T_\pm} \|u(t)\|_{H^s} = \infty$ for all $s > 0$.
    As part of our proof, we demonstrate an orbital stability statement for the soliton and a dispersive decay bound for solutions with suitable initial data.
\end{abstract}

\section{Introduction}

This paper is devoted to the focusing {\em continuum Calogero--Moser model}
\begin{equation}\tag{CCM}\label{CM-DNLS}
    iu_t + u_{xx} - 2 i u \Pi^+ \partial_x(|u|^2) = 0,
\end{equation}
known in some sources as the {\em Calogero--Moser derivative nonlinear Schr\"odinger equation}. Here $\Pi^+$ is the Cauchy--Szeg\H{o} projection; see \eqref{pi_+}. Solutions are complex-valued functions of spacetime that belong to a Hardy--Sobolev space
\begin{equation}\label{Hardy Sobolev}
    H^s_+(\R) := \{f \in H^s(\R) : \widehat{f}(\xi) = 0 \text{ for } \xi < 0\}, \quad s \geq 0,
\end{equation}
for each time $t$.


A defocusing version of \eqref{CM-DNLS} appeared in \cite{pelinovsky} as a model for internal waves in a stratified fluid, with the nonlinear Schr\"odinger equation emerging as a small-depth limit.
The focusing version studied here is derived in \cite{abanov} as a formal hydrodynamic limit of the Calogero-Moser system --- a Hamiltonian system of particles on a line interacting with an inverse square potential:
\begin{equation}\tag{CM}\label{CM}
    V(x_j) = \sum_{k \neq j} \frac{1}{(x_j - x_k)^2}.
\end{equation}
This particle system, and its generalizations, are well-studied; it is solved in the quantum mechanical case \cite{calogero,calogero2} and is integrable in the classical case \cite{moser}.
The discrete periodic case was studied by Sutherland \cite{sutherland,sutherland2} and the periodic continuum limit has been studied under the name of {\em Calogero--Sutherland derivative NLS} in \cite{badreddine,badreddine2}.


As a continuum limit of an integrable system, it is unsurprising that \eqref{CM-DNLS} is completely integrable. It has a Lax pair given by
\begin{equation*}
    \mathcal{L}_u = -i \partial_x - u \Pi^+ \overline{u} \qtq{and} \mathcal{P}_u = i \partial_x^2 + 2u \Pi^+ \partial_x \overline{u}
\end{equation*}
where $u,\overline{u}$ are understood as multiplication by $u,\overline{u}$, respectively; see \cite{gerard-lenzmann}.
This Lax pair generates an infinite hierarchy of conserved quantities \cite{gerard-lenzmann}
\begin{equation}\label{conserved}
    I_k(u) := \langle u, \mathcal{L}_u^k u \rangle.
\end{equation}
The first and most relevant of these are the mass
\begin{equation*}
    M(u) := I_0(u) = \int|u|^2,
\end{equation*}
the momentum
\begin{equation*}
    P(u) := I_1(u) = \int -i\overline{u} u_x - \tfrac{1}{2} |u|^4,
\end{equation*}
and the Hamiltonian
\begin{equation*}
    E(u) := \tfrac{1}{2} I_2(u) = \tfrac{1}{2} \int \big|u_x - i u \Pi^+(|u|^2)\big|^2.
\end{equation*}
This Hamiltonian generates the flow of \eqref{CM-DNLS}, albeit with a nonstandard Poisson structure arising from a gauge transformation in \cite{abanov}.


Given the infinite family of conserved quantities \eqref{conserved}, one might expect that solutions to \eqref{CM-DNLS} satisfy a priori bounds in $H^s$ for all $s \geq 0$.
Indeed, for initial data $u_0 \in H^\infty_+$ with $M(u_0) < 2\pi$,
G\'{e}rard and Lenzmann showed in \cite{gerard-lenzmann} that the corresponding maximal lifespan solution $u(t)$ satisfies
\begin{equation}\label{a priori bounds}
    \|u\|_{L^\infty_tH_x^n} \leq C(I_0(u), I_1(u), ..., I_{2n}(u))
\end{equation}
for all integers $n \geq 0$.
Using \eqref{a priori bounds} they showed that for initial data with mass less than $2\pi$, \eqref{CM-DNLS} is globally well-posed in $H^n_+$ for all integers $n \geq 1$.

On the other hand, G\'{e}rard and Lenzmann also showed that \eqref{CM-DNLS} has multisoliton solutions for which these a priori bounds fail.
Specifically, for $N \geq 2$, the $N$-soliton solutions, which have mass $2 \pi N$, exhibit a growth of Sobolev norms of the form $\|u(t)\|_{H^s} \sim |t|^{2s}$ for large $|t|$.
This raised the question of what is the true threshold for the growth of Sobolev norms.
We answer this question by proving the following:
\begin{theorem}[Turbulent behaviour]\label{main result}
    For any sufficiently small $\epsilon > 0$, there exist initial data $u_0 \in H^\infty_+(\R)$ with $M(u_0) = 2\pi + \epsilon$, a time $T \in (0, \infty]$, and a maximal lifespan solution $u(t)$ to \eqref{CM-DNLS} such that for all $s > 0$, $u \in C_tH^s_x([0,T) \times \R)$ and
    \begin{equation*}
        \lim_{t \nearrow T} \|u(t)\|_{H^s} = +\infty.
    \end{equation*}
    In particular, if $T = \infty$, then we have the bounds
    \begin{equation*}
        \|u(t)\|_{H^s} \gtrsim t^s.
    \end{equation*}
    An analogous statement holds backwards in time.
\end{theorem}

This theorem shows that the global-in-time $H^s$ bounds \eqref{a priori bounds} fail above the mass of $2\pi$.
This presents a significant obstacle to existing techniques for proving global well-posedness of \eqref{CM-DNLS} above the mass of $2\pi$.

The $2\pi$ mass threshold appears in this analysis through the sharp inequality 
\begin{equation}\label{mass bound}
    \|\Pi^+ \overline{u} f\|_{L^2}^2 \leq \tfrac{1}{2\pi} M(u) \langle f, -i\partial_x f\rangle
\end{equation}
valid for $u \in L^2_+$ and $f \in H^{\frac{1}{2}}_+$; see \cite{gerard-lenzmann}.
This inequality is saturated by the function
\begin{equation*}
    Q(x) = \frac{\sqrt{2}}{x + i}
\end{equation*}
which is the unique nonzero optimizer up to phase, translation, and $L^2$-scaling symmetries.
It is known as the soliton solution to \eqref{CM-DNLS} and does not evolve in time.
The soliton satisfies $M(Q) = 2\pi$ and $I_k(Q) = 0$ for all $k \geq 1$.

Within our proof, this mass threshold also occurs prominently in the existence of eigenvalues for the Lax operator. In \cite{gerard-lenzmann}, it was shown that
\begin{equation}\label{eigenvalue bound}
    2 \pi N \leq M(u)
\end{equation}
where $N$ is the number of eigenvalues of $\mathcal{L}_u$. Our proof requires initial data $u_0$ for which the Lax operator $\mathcal{L}_{u_0}$ has a negative eigenvalue, which is consistent with our solutions having mass at least $2\pi$.
The soliton $Q$ and the multisolitons are introduced in \cite{gerard-lenzmann} as optimizers of \eqref{eigenvalue bound}.


Let us briefly review the available well-posedness results on this model. Local well-posedness in $H^s_+(\R)$ for $s > \tfrac{1}{2}$ was established by G\'erard and Lenzmann in \cite{gerard-lenzmann}.
This was shown for $s > \frac{3}{2}$ via a Kato-type iterative scheme and standard energy methods and extended to $s > \frac{1}{2}$ following an argument in \cite{lower-regularity}.
For $M(u) < 2\pi$, G\'erard and Lenzmann used \eqref{mass bound} to control $\|u\|_{H^{\frac{1}{2}}}$ in terms of $P(u)$ and $M(u)$.
They obtained the bounds \eqref{a priori bounds} on $H^s$ for all $s \geq 0$ iteratively using the conserved quantities $I_k(u)$ defined in \eqref{conserved}.
This extended local well-posedness in $H^s_+$ to global well-posedness in $H^n_+$ for integers $n \geq 1$, so long as $M(u) < 2\pi$. 
They also showed that $H^s_+$ solutions for $s \geq 1$ to \eqref{CM-DNLS} with mass exactly $2\pi$ cannot blow up in finite time, which extended global well-posedness to mass at most $2\pi$ for such regular solutions.

Recent work by Laurens, Killip, and Vi\c{s}an \cite{monica-rowan-thierry-2023} extended global well-posedness below the mass of $2\pi$ to $H^s_+(\R)$ for $s \geq 0$.
Central to their analysis was the concept of $L^2$-equicontinuity:
\begin{definition}[$L^2$ Equicontinuity]\label{equicontinuity}
    A bounded set $U \subset L^2(\R)$ is said to be {\em equicontinuous} in the $L^2(\R)$ topology if
    \begin{equation*}
        \limsup_{\delta\to0}\sup_{\substack{u \in U,\; |y| < \delta}} \|u(\cdot +y) - u(\cdot)\|_{L^2} = 0,
    \end{equation*}
    or equivalently,
    \begin{equation*}
        \limsup_{\kappa \to \infty} \,\sup_{u \in U} \int_{|\xi| \geq \kappa} |\widehat{u}(\xi)|^2 d\xi = 0.
    \end{equation*}
\end{definition}
The argument in \cite{monica-rowan-thierry-2023} relied on showing that any $L^2$-equicontinuous family $U \subset H^\infty_+$ of initial data leads to a family of partial orbits
\begin{equation*}
    U_T^* = \{u(t) : u \text{ solves \eqref{CM-DNLS}}, u(0) \in U, t \in [0, T)\}
\end{equation*}
which is $L^2$-equicontinuous for any choice of $T > 0$.
They defined the sharp mass threshold $M^*$ such that if $U \subset H^\infty_+$ is $L^2$-equicontinuous and $\sup_{u_0 \in U} M(u_0) < M^*$, then $U_T^*$ is also $L^2$-equicontinuous for all $T > 0$.
They showed that $M^* \geq 2\pi$ and that \eqref{CM-DNLS} is globally well-posed in $H^s_+$ for all $s \geq 0$ below the mass $M^*$.
The exact value of $M^*$ is currently unknown, but Theorem~\ref{main result} shows that global-in-time $L^2$-equicontinuity fails in general for mass greater than $2\pi$.
Specifically, for any $\epsilon > 0$ sufficiently small, we find initial data $u_0 \in H^\infty_+$ with $M(u_0) = 2\pi + \epsilon$ for which the maximal orbit fails to be $L^2$-equicontinuous.

The local theory in \cite{gerard-lenzmann} allows us to consider $H^\infty_+$ solutions of \eqref{CM-DNLS}, which are sufficient for our analysis.
For $H^\infty_+$ solutions, Laurens, Killip, and Vi\c{s}an showed in \cite{monica-rowan-thierry-2023} that finite-time blowup must be accompanied by a loss of $L^2$-equicontinuity and hence simultaneous blowup in $H^s$ for all $s > 0$.
The well-posedness necessary for our main theorem is then stated as follows:
\begin{proposition}\label{well-posedness}
    For initial data $u_0 \in H^\infty_+$ there exists a maximal time $T \in (0, \infty]$ and a unique solution $u(t)$ such that $u \in C_t H^s_+ ([0, T) \times \R)$ for all $s \geq 0$.
    Moreover, if $T < \infty$ then
    \begin{equation*}
        \lim_{t \nearrow T} \|u(t)\|_{H^s} = +\infty
    \end{equation*}
    for all $s > 0$.
    An analogous statement holds backwards in time.
\end{proposition}
    \subsection{Overview of proof}
    We start by proving an orbital stability result for the soliton for a restricted class of initial data in Theorem~\ref{stability result}.
Specifically, we show that for any $\epsilon > 0$, if $u_0$ is sufficiently close to $Q$ in $H^1$ and $u(t)$ is uniformly bounded below in $\dot{H}^1$, then
\begin{equation*}
    \inf_{\lambda > 0; \;\theta, y \in \R} \|e^{i\theta} \lambda^{\frac{1}{2}} u(t, \lambda x + y) - Q(x)\|_{H^1_x} < \epsilon,
\end{equation*}
for all times of existence $t$. 
This allows us to choose parameters $\lambda(t) > 0$ and $\theta(t), y(t) \in \R$ such that
\begin{equation}\label{modulation equation}
    \|e^{i\theta(t)} \lambda(t)^{\frac{1}{2}} u(t, \lambda(t) x + y(t)) - Q(x)\|_{H^1_x} < \epsilon.
\end{equation}
The scaling parameter $\lambda(t)$ then represents the characteristic width of $u(t)$.
We verify that this set of solutions is non-empty by showing that initial data $u_0 \in H^\infty_+$ for which $\mathcal{L}_{u_0}$ has a negative eigenvalue lead to solutions that satisfy a uniform lower bound in $\dot{H}^1$; see Lemma~\ref{lower bound}.
An example of such initial data is $u_0 = (1+\epsilon)Q$ for $\epsilon > 0$ as shown in Lemma~\ref{negative eigenvalue}.

We then prove a dispersive decay bound in Theorem~\ref{dispersive decay} that shows that global-in-time $H^\infty_+$ solutions to \eqref{CM-DNLS} satisfy $|u(t, z)| \to 0$ for $\Im z > 0$ as $t \to \infty$, where $u(t, z)$ is defined by holomorphic extension to the upper half-plane.
For solutions satisfying \eqref{modulation equation}, this dispersive decay is inconsistent with $\lambda(t) \sim 1$ uniformly in $t$; we must either have that $\lambda(t) \to 0$ or $\lambda(t) \to \infty$ as $t \to \infty$.
The lower bound on $\|u(t)\|_{\dot{H}^1}$ rules out the case of $\lambda(t) \to \infty$ and so we must have $\lambda(t) \to 0$ as $t \to \infty$.
This forces $\|u(t)\|_{H^s}$ to grow to $\infty$ as $t \to \infty$, as desired.
The lower bound $\|u(t)\|_{H^s} \gtrsim t^s$ then follows from the rate of dispersive decay that we obtain.
    \subsection{Acknowledgements}
    J.H.\ and M.K.\ were supported in part by NSF grant DMS-2154022 and NSF grant DMS-2054194. We would like to thank our advisors, Rowan Killip and Monica Vi\c{s}an, for their guidance and support.
    \subsection{Notation}

We use the standard notation $A \lesssim B$ to indicate that $A \leq C B$ for some universal constant $C > 0$. If both $A \lesssim B$ and $B \lesssim A$ then we use the notation $A \sim B$.
When the implied constant fails to be universal, the relevant dependencies will be indicated within the text.

Our conventions for the Fourier transform are
\begin{equation*}
\widehat{f}(\xi) = \tfrac{1}{\sqrt{2\pi}}\int e^{-i \xi x} f(x) dx \quad \text{so} \quad f(x) = \tfrac{1}{\sqrt{2\pi}} \int e^{i \xi x} \widehat{f}(\xi) d\xi.
\end{equation*}
These Fourier transforms are unitary on $L^2$ and hence yield the standard Plancherel identities. With these conventions, we define the Cauchy--Szeg\"o projector $\Pi^+ : L^2 \to L^2_+$ as
\begin{equation}\label{pi_+}
[\Pi^+ f](x) = \tfrac{1}{\sqrt{2\pi}} \int_0^\infty e^{ix\xi} \widehat{f}(\xi) d\xi.
\end{equation}


For $s \geq 0$, we define the Sobolev spaces $H^s$ as the completion of $\mathcal{S}(\R)$ with respect to the norm
\begin{equation*}
    \|f\|^2_{H^s} = \int \big( 1 + |\xi|^2 \big)^s |\widehat{f}(\xi)|^2 d\xi.
\end{equation*}
The homogeneous Sobolev spaces $\dot{H}^s$ are defined similarly by the norm
\begin{equation*}
    \|f\|^2_{\dot{H}^s} = \int |\xi|^{2s} |\widehat{f}(\xi)|^2 d\xi.
\end{equation*}
We work in the Hardy-Sobolev spaces, denoted $H_+^s = H_+^s(\R) \subset H^s(\R)$, comprised of those functions whose Fourier transforms are supported on $[0,\infty)$ as in \eqref{Hardy Sobolev}. Equivalently, $H^s_+$ is comprised of $f\in H^s$ whose Poisson integral
\begin{equation*}\label{harmonic extension}
    f(z) = \int\frac{\Im z}{\pi|x- z|^2} f(x) dx \quad \text{defined for} \quad \Im z > 0
\end{equation*}
is holomorphic in the upper half-plane. This formulation of $H^s_+$ allows us to work interchangeably with $f(x)$ for $x \in \R$ and $f(z)$ for $\Im z > 0$. With the Poisson formulation, H\"older's inequality implies
\begin{equation}\label{continuity of Hardy extension}
    |f(z)| \lesssim (\Im z)^{-1/p} \|f\|_{L^p(\R)}
\end{equation}
for all $1 \leq p \leq \infty$. In this paper, we mainly work in the spaces $L^2_+ = H^0_+$ and $H^\infty_+ = \cap_{s} H^s_+$.

We observe that \eqref{CM-DNLS} is invariant under phase rotation, translation, and $L^2$-critical scaling
\begin{equation*}
    u(t,x) \mapsto e^{i\theta} \lambda^{\frac{1}{2}} u(\lambda^2 t, \lambda x + y)
\end{equation*}
for $\lambda > 0$ and $y,\theta \in \R$. We adopt the notation
\begin{equation}\label{symmetries}
    u_{\lambda, y, \theta}(t,x) = e^{i\theta} \lambda^{\frac{1}{2}} u(t, \lambda x + y)
\end{equation}
for $\lambda > 0$ and $y,\theta \in \R$. Note that only the spatial component is altered here.

\section{Orbital stability of the soliton}
In this section we seek to prove

\begin{theorem}[Restricted orbital stability of $Q$]\label{stability result}
    Fix $c > 0$. For all $\epsilon > 0$ there exists $\delta > 0$ such that if $u_0 \in H^1$ satisfies
    \begin{equation*}
        \|u_0 - Q\|_{H^1} < \delta
    \end{equation*}
    and the corresponding maximal lifespan solution $u$ satisfies $\|u(t)\|_{\dot{H}^1} \geq c$, then
    \begin{equation*}
        \inf_{\lambda > 0; \;\theta, y \in \R}\|u_{\lambda, \theta, y}(t)-Q\|_{H^1} < \epsilon
    \end{equation*}
    for all times of existence $t$, using the notation from \eqref{symmetries}.
\end{theorem}
Note that in our orbital distance, the symmetries are applied to $u$ rather than $Q$. This is reminiscent of modulation theory for gKdV and NLS, \cite{martel-merle-2001,martel-merle-2002, weinstein1985modulational}.
In the case of $s = 0$, scaling-criticality implies 
\begin{equation*}
    \inf_{\lambda>0; \; \theta, y \in \R} \|u_{\lambda, \theta, y} - Q\|_{L^2} = \inf_{\lambda>0; \; \theta, y \in \R} \|u - Q_{\lambda, \theta, y}\|_{L^2},
\end{equation*}
which is the standard distance from $u$ to the manifold of solitons. In contrast, for $s \neq 0$, the scaling term prevents the symmetries from being moved to $Q$ without accruing additional factors of $\lambda$. Because of this, our orbital distance for $s \neq 0$ differs from the standard notion.

We argue via a variational approach akin to Cazenave and Lions \cite{cazenave-lions}. We therefore recall the characterization of $Q$ as a minimizer of the energy, as given by G\'erard and Lenzmann in \cite{gerard-lenzmann}.
\begin{lemma}[Minimal mass bubble, \cite{gerard-lenzmann}]\label{Minimal Mass Bubble Lemma}
    Suppose that $(v^n)_{n \in \N} \subset H^1$ is a $L^2$-bounded sequence which satisfies
    \begin{align*}
        \lim_{n \to \infty} E(v^n) = 0, \;\;\;
        \liminf_{n \to \infty} \|v^n\|_{L^2}^2 = \|Q\|_{L^2}^2, \;\;\; \text{and} \;\;\;
        \|v^n\|_{\dot{H}^1} = \mu > 0\quad \text{for all } n \in \N.
    \end{align*}
    Then after possibly passing to a subsequence,
    \begin{equation*}
        \|v^n_{\lambda, \theta, x_n} - Q\|_{L^2} \to 0
    \end{equation*}
    for some constants $\lambda > 0$, $\theta \in \R$, and some sequence $(x_n) \subset \R$.
\end{lemma}

As $Q$ is the unique energy minimizer up to symmetries (see \cite{gerard-lenzmann}), the preceding lemma informally asserts that convergence of mass and energy to those of $Q$ implies convergence in $L^2$ to $Q$ up to symmetries. Since mass and energy are conserved and $Q$ is stationary, this suggests the flow of \eqref{CM-DNLS} should preserve closeness to $Q$ in the $L^2$ topology. To upgrade this to the $H^1$ topology, we use the conservation of energy. This is formalized in what follows.

\begin{proof}[Proof of Theorem~\ref{stability result}]
    We argue by contradiction. Suppose that there exist some $\epsilon > 0$ and a sequence of initial data $u^n_0\to Q$ in $H^1$ with maximal lifespan solutions $u^n$ such that $\|u^n(t)\|_{\dot{H}^1} \geq c$ and 
    \begin{equation*}
        \inf_{\lambda > 0; \;  \theta, y \in \R}\|u^n_{\lambda, \theta, y}(t_n)-Q\|_{H^1} \geq \epsilon
    \end{equation*}
    for some time of existence $t_n$. 
    
    To reach a contradiction, we aim to apply Lemma~\ref{Minimal Mass Bubble Lemma} to $u^n(t_n)$. We first normalize $u^n(t_n)$ in $\dot{H}^1$ and define
    \begin{equation*}
         \lambda_n = \|u^n(t_n)\|_{\dot{H}^1}^{-1} \leq c^{-1} \qtq{and} v^n(x) = u^n_{\lambda_n,0,0}(t_n).
    \end{equation*}
    We note that
    \begin{equation*}
        \sup_n \|v^n\|_{L^2} = \sup_n \|u^n_0\|_{L^2} < \infty \qtq{and} \|v^n\|_{\dot{H}^1} = \lambda_n \|u^n(t_n)\|_{\dot{H}^1} = 1.
    \end{equation*}
    Furthermore,
    \[0 \leq E(v^n) = \lambda_n^2 E(u_0^n) \leq c^{-2} E(u_0^n) \to 0\]
    because $u_0^n \to Q$ in $H^1$.
    Passing to a subsequence, Lemma~\ref{Minimal Mass Bubble Lemma} then guarantees that
    \begin{equation*}
        \|v^n_{\lambda, \theta, x_n} - Q\|_{L^2} \to 0
    \end{equation*}
    for some $\lambda > 0$, $\theta \in \R$, and a sequence $(x_n) \subset \R$. 

    Next, we upgrade this $L^2$-convergence to $H^1$-convergence using an energy argument. Define
        \[w^n = v^n_{\lambda, \theta, x_n} - Q.\]
    To show that $v^n_{\lambda, \theta, x_n} \to Q$ in $H^1$ it then suffices to show that $\|w^n\|_{\dot{H}^1} \to 0$.
    Because
    \[\mathcal{L}_QQ = -iQ_x - \Pi_+(|Q|^2)Q = 0\]
    we compute that
    \begin{align*}
        E(v^n_{\lambda, \theta, x_n})
        & = \tfrac{1}{2} \int |\partial_x w^n - i \Pi_+ (|Q|^2)w^n - i \Pi_+(|w^n|^2 + 2 \real \overline{Q} w^n)(Q + w^n)|^2\\
        & = \tfrac{1}{2}\int |\partial_x w^n|^2 + O(\|w^n\|_{L^2} + \|w^n\|_{L^2}^2).
    \end{align*}
    We note that the preceding calculations repeatedly use the fact that $\|w^n\|_{L^\infty} \lesssim 1$ uniformly in $n$, which follows from Sobolev embedding.
    
    Because $w^n \to 0$ in $L^2$ and $E(v^n_{\lambda, \theta, x_n}) = \lambda^2E(v^n) \to 0$, it then follows that $\|w^n\|_{\dot{H}^1} \to 0$ and so $v^n_{\lambda, \theta, x_n} \to Q$ in $H^1$.
    This implies that 
        \[\lim_{n \to \infty} \inf_{\lambda > 0; \;  \theta, y \in \R}\|u^n_{\lambda, \theta, y}(t_n)-Q\|_{H^1} = 0\]
    which is a contradiction.
\end{proof}

To verify that this stability result is not vacuous, we show that a solution $u(t)$ for which $\mathcal{L}_u$ has a negative eigenvalue satisfies $\|u(t)\|_{\dot{H}^1} \gtrsim 1$ uniformly over the times of existence.
As $\mathcal{L}_u, \mathcal{P}_u$ form a Lax pair, we expect that the map $t \mapsto \mathcal{L}_{u(t)}$ is an isospectral deformation and thus that the negative eigenvalue of $\mathcal{L}_u$ is preserved in time.
This was shown in \cite{moser} in the case that $\mathcal{L}_u$ and $\mathcal{P}_u$ are matrices and applied to \eqref{CM}. 
To extend this result to the case where $\mathcal{L}_u$ and $\mathcal{P}_u$ are unbounded operators, more careful analysis is required.
For an example of this analysis for \eqref{CM-DNLS}, see \cite{monica-rowan-thierry-2023}, where it is shown that $t \mapsto \mathcal{L}_{u(t)}$ is an isospectral deformation provided that $u(t) \in H^\infty_+$.

\begin{lemma}[Uniform lower bound in $\dot{H}^1$]\label{lower bound}
    Suppose there exist $\epsilon > 0$ and $u_0 \in H^\infty_+$ such that $\|u_0\|_{L^2}^2 \leq 2\pi + \epsilon$ and such that $\mathcal{L}_{u_0}$ has an eigenvalue in $(-\infty, -c\epsilon]$ for some $c > 0$.
    Let $u$ denote the corresponding maximal lifespan solution to \eqref{CM-DNLS}.
    Then
        \[ \|u(t)\|_{\dot{H}^1} \gtrsim c\]
    uniformly for all times of existence $t$.
\end{lemma}
\begin{proof}
    Since $\mathcal{L}_{u(t)}$ is isospectral, there exists an eigenvector $\psi(t) \in H^{\frac{1}{2}}_+$ for all times of existence $t$ with $\|\psi(t)\|_{L^2} = 1$ such that
        \[\langle \psi(t), \mathcal{L}_{u(t)} \psi(t) \rangle \leq -c\epsilon.\]
    Expanding $\mathcal{L}_{u(t)}$ yields
    \begin{align*}
        -c\epsilon & \geq 
        \langle \psi(t), -i \partial_x\psi(t)\rangle - \Bigg[\frac{2\pi}{\|u(t)\|_{L^2}^2} + \frac{\|u(t)\|_{L^2}^2 - 2\pi}{\|u(t)\|_{L^2}^2}\Bigg]\|\Pi^+ \overline{u(t)} \psi(t)\|_{L^2}^2.
    \end{align*}
    Inequality \eqref{mass bound} then implies
    \begin{align*}
        c\epsilon &\leq \frac{\|u(t)\|_{L^2}^2 - 2\pi}{\|u(t)\|_{L^2}^2}\|\Pi^+ \overline{u(t)} \psi(t)\|_{L^2}^2 \\
        &\leq \frac{\epsilon}{2\pi + \epsilon} \|u(t)\|_\infty^2 \\
        &\leq \epsilon \|u(t)\|_\infty^2.
    \end{align*}
    Therefore $c \leq \|u\|_\infty^2$ and so $c \lesssim \|u\|_{L^2} \|u\|_{\dot{H}^1} \lesssim \|u\|_{\dot{H}^1}$ as desired.
\end{proof}

We now show that initial data of the form $u_0 = (1+\epsilon)Q$ for $\epsilon > 0$ has a uniformly negative eigenvalue. This data will be the starting point for our proof of the turbulent threshold.
\begin{lemma}[Existence of negative eigenvalue]\label{negative eigenvalue}
    For any $\epsilon > 0$, the Lax operator $\mathcal{L}_{(1+\epsilon)Q}$ has a negative eigenvalue in $(-\infty, -\epsilon]$.
\end{lemma}
\begin{proof}
    Because the Lax operator is self-adjoint,
    \begin{align*}
        2\pi \inf \sigma (\mathcal{L}_{(1+\epsilon)Q}) & = \min_{\psi \neq 0} \frac{2\pi}{\|\psi\|_{L^2}^2}\left\langle \psi, \mathcal{L}_{(1+\epsilon)Q}\psi\right\rangle \\
        & \leq \left\langle Q, \mathcal{L}_{(1+\epsilon)Q}Q\right\rangle \\
        &  = \left\langle Q,-i\partial_xQ\right\rangle - (1+\epsilon)^2\|\Pi_+|Q|^2\|_{L^2}^2 \\
        & = \left\langle Q, \mathcal{L}_{Q}Q\right\rangle - (2\epsilon + \epsilon^2)\pi \\
        & \leq - 2\pi \epsilon,
    \end{align*}
    where the final equality follows from $\mathcal{L}_QQ=0$. Therefore $\inf \sigma (\mathcal{L}_{(1+\epsilon)Q}) \leq - \epsilon$. Because $\sigma_{\text{ess}}(\mathcal{L}_{(1+\epsilon)Q}) = [0,\infty)$ (see \cite{gerard-lenzmann,monica-rowan-thierry-2023}), this concludes the proof.
\end{proof}

\section{Dispersive decay}\label{section: dispersive decay}

Our goal in this section is to establish the following dispersive decay estimate for solutions to \eqref{CM-DNLS} with suitable initial data:

\begin{theorem}[Dispersive decay]\label{dispersive decay}
Given a set of initial data $U \subset H^\infty_+$ which is bounded and equicontinuous in $L^2_+$ and satisfies $\langle x \rangle u_0 \in L^2$ for all $u_0 \in U$,
\begin{equation}
    |u(t,z)| \lesssim |t|^{-\frac{1}{2}}\|u_0\|_{L^1}\big[1 + \|u_0\|_{L^2}^2(1 + (\Im z)^{-1})\big]
\end{equation}
uniformly for $\Im z > 0$, $u_0 \in U$, and all times of existence $t$.
\end{theorem}


Using Theorem~\ref{stability result}, we will consider solutions $u(t)$ for which
\begin{equation*}
    \|u_{\lambda(t),\theta(t),y(t)}(t) - Q\|_{H^1} < \epsilon,
\end{equation*}
for some $\lambda(t) > 0$ and $\theta(t), y(t) \in \R$. Theorem~\ref{dispersive decay} then implies that the characteristic width $\lambda(t)$ of $u(t)$ cannot satisfy $\lambda(t) \sim 1$ uniformly in $t$.
This deviates from many classical stability results for which the scaling remained constant; e.g.\ \cite{weinsteinlyapunov}.


To prove Theorem~\ref{dispersive decay}, we use the explicit formula established in \cite{monica-rowan-thierry-2023} for $H^\infty_+$ solutions to \eqref{CM-DNLS}.
This explicit formula is of the type developed by G\'erard and collaborators; see \cite{gerard2022explicit,gerard2015explicit,gerard2023cubic}.
A similar explicit formula was established for an intertwined matrix version of \eqref{CM-DNLS} in \cite{sun}.
The explicit formula for \eqref{CM-DNLS} was also shown to hold for initial data in $L^2_+$ below the mass of $2\pi$ in \cite{monica-rowan-thierry-2023}, which can be used to expand the application of Theorem~\ref{dispersive decay}. As we are largely interested in $H^\infty_+$ solutions, we recall the explicit formula in that case.
\begin{theorem}[Explicit formula]\label{explicit formula}
For any $H^\infty_+(\R)$ solution $u(t)$ to \eqref{CM-DNLS} with initial data satisfying $\langle x\rangle u_0 \in L^2$,
\begin{equation*}
	u(t,z) = \tfrac{1}{2\pi i} I_+\Big\{(X + 2 t \mathcal{L}_{u_0} - z)^{-1} u_0\Big\}
\end{equation*}
for all $\Im z > 0$ and all times of existence $t$.
\end{theorem}


Here $X$ and $I_+$ are defined as follows:
\begin{definition}
The unbounded linear functional $I_+$ is defined by
\begin{equation*}
    I_+(f) := \lim_{y \to \infty} \sqrt{2\pi} \int_0^\infty y e^{-y\xi} \widehat{f}(\xi) d\xi
\end{equation*}
with domain $D(I_+)$ given by the set of those $f \in L^2_+(\R)$ for which the limit exists.
\end{definition}

\begin{definition}
The unbounded operator $X : L^2_+ \to L^2_+$ is given by
\begin{equation*}
    \widehat{Xf}(\xi) := i \tfrac{d\widehat{f}}{d\xi}(\xi) \qtq{for} f \in D(X)
\end{equation*}
where $D(X)$ is the domain $D(X) = \{f \in L^2_+(\R) : \widehat{f} \in H^1([0,\infty])$.
\end{definition}


To simplify notation, let
\begin{equation*}
	A(t, z; u) := (X + 2t\mathcal{L}_u - z)^{-1},
\end{equation*}
so that the explicit formula in Theorem~\ref{explicit formula} becomes
\begin{equation*}
	u(t, z) = \tfrac{1}{2\pi i}I_+ A(t, z; u_0)u_0.
\end{equation*}
Using the resolvent identity as was done in \cite{monica-rowan-thierry-2023}, we may write
\begin{equation}\label{resolvent formula}
	A(t, z; u_0) = A_0(t, z) + 2tA_0(t, z) u_0 \Pi^+ \overline{u_0} A(t, z; u_0)
\end{equation}
where $A_0(t,z) = A(t,z;0)$.


We will employ the following bounds on $A_0$ and $A$:

\begin{proposition}[Inequalities (4.21) and (4.41) from \cite{monica-rowan-thierry-2023}]\label{A and A_0 bounds}
The operator $I_+ A_0$ is bounded on $L^1$, namely
\begin{equation*}
	|I_+ A_0(t, z) f| \lesssim |t|^{-\frac{1}{2}} \|f\|_{L^1},
\end{equation*}
uniformly for $t \in \R$ and $\Im z > 0$.

Given a bounded and equicontinuous set $U \subset L^2_+(\R)$, the bound
\begin{equation*}
	\|A(t, z; u) f\|_{L^\infty} \lesssim |t|^{-1} \big[1 + (\Im z)^{-1}\big] \|f\|_{L^1}
\end{equation*}
holds uniformly for $u\in U$, $t \in \R$, $f \in L^1$, and $\Im z > 0$.
\end{proposition}

Combining these, we establish a dispersive decay for $H^\infty_+$ solutions.


\begin{proof}[Proof of Theorem~\ref{dispersive decay}]
Fix $u_0 \in U$.
A simple Gr\"onwall argument ensures that $\langle x \rangle u(t) \in L^2$ for all times of existence and H\"older's inequality then guarantees that $U \subset L^1$.
Combining the explicit formula in Theorem~\ref{explicit formula}, the resolvent identity \eqref{resolvent formula}, and the bounds from Proposition~\ref{A and A_0 bounds}, we may estimate
\begin{align*}
    |u(t,z)| & \lesssim \left|I_+ \{A_0(t,z) u_0\}\right| + |t|\left|I_+\{A_0(t,z) u_0 \Pi^+ \overline{u_0} A(t,z;u_0) u_0\}\right| \\
			& \lesssim |t|^{-\frac{1}{2}}\|u_0\|_{L^1} + |t|^{\frac{1}{2}}\|u_0 \Pi^+ \overline{u_0} A(t,z;u_0) u_0\|_{L^1} \\
			&\lesssim |t|^{-\frac{1}{2}}\|u_0\|_{L^1}\big[1 + \|u_0\|_{L^2}^2 (1 + (\Im z)^{-1})\big]
\end{align*}
uniformly for $\Im z > 0$, $u_0 \in U$, and all times of existence $t$.
\end{proof}

\section{Growth of Sobolev norms}
In this section, we find $H^\infty_+$ solutions arbitrarily close to, but above, the mass threshold of $2\pi$ that exhibit global Sobolev growth:
\begin{theorem}[Turbulent behaviour]\label{turbulent behavior}
    For any sufficiently small $\epsilon > 0$, there exist initial data $u_0 \in H^\infty_+(\R)$ with $M(u_0) = 2\pi + \epsilon$, a time $T \in (0, \infty]$, and a solution $u(t)$ to \eqref{CM-DNLS} such that for all $s > 0$, $u \in C_tH^s_x([0,T) \times \R)$ and
    \begin{equation*}
        \lim_{t \nearrow T} \|u(t)\|_{H^s} = +\infty.
    \end{equation*}
    In particular, if $T = \infty$, then we have the bounds
    \begin{equation*}
        \|u(t)\|_{H^s} \gtrsim t^s.
    \end{equation*}
    An analogous statement holds backwards in time.
\end{theorem}

Our initial data will be a perturbation of the soliton $Q$ belonging to $\mathcal{S}(\R) \cap L^2_+$.
This is more than sufficient to satisfy the regularity constraints arising in Theorem~\ref{dispersive decay}.
Our proof will require that the implicit constants which come from Theorem~\ref{dispersive decay} are uniform; see the progression from \eqref{lambda bounds 0} to \eqref{lambda bounds 1}.
We therefore construct an $L^2$-equicontinuous family of initial data as follows:
\begin{lemma}\label{initial data}
    There exists an $L^2$-equicontinuous family $U = \{u_0^\epsilon : 0 < \epsilon \leq 1\} \subset L^2_+ \cap \mathcal{S}(\R)$ such that $\|u_0^\epsilon - Q\|_{L^2} \lesssim \epsilon$ and the corresponding solutions $u^\epsilon(t)$ satisfy
    \begin{equation*}
        \|u^\epsilon(t)\|_{\dot{H}^1} \gtrsim 1 \qtq{and} \inf_{\lambda > 0; \; \theta, y \in \R}\|u_{\lambda, \theta, y}^\epsilon(t) - Q\|_{H^1} < \epsilon
    \end{equation*}
    uniformly for all times of existence $t$ and initial data $u^\epsilon_0 \in U$.
\end{lemma}

\begin{proof}
    To define $u_0^\epsilon$, we begin with $(1+\epsilon)Q$ which has a negative eigenvalue in $(-\infty, -\epsilon]$ as shown in Lemma~\ref{negative eigenvalue}.
    We recall that $Q$ has the Fourier transform
    \begin{equation}\label{Q hat}
        \widehat{Q}(\xi) = 2 \sqrt{\pi} i e^{-\xi} \chi_{\{\xi \geq 0\}}.
    \end{equation}
    Therefore the only obstruction to $(1+\epsilon)Q$ belonging to $\mathcal{S}(\R)$ is that $(1+\epsilon)\widehat{Q}$ has a discontinuity at $\xi = 0$.
    
    We define $u^\epsilon_0 \in \mathcal{S}\cap L^2_+$ to be a smoothing of $(1+\epsilon)Q$ on the Fourier side which removes the discontinuity at $\xi = 0$. This can be done with sufficiently small perturbation in $H^1$ to ensure that $\mathcal{L}_{u^\epsilon_0}$ has a negative eigenvalue in $(-\infty, \frac{\epsilon}{2}]$.
     
    Let $u^\epsilon(t)$ denote the maximal solution to \eqref{CM-DNLS} with initial data $u_0^\epsilon$. The bounds $\|u^\epsilon(t)\|_{\dot{H}^1} \gtrsim 1$ and $\inf_{\lambda > 0;\theta, y \in \R}\|u_{\lambda, \theta, \epsilon}^\epsilon(t) - Q\|_{H^1} < \epsilon$ then follow from Theorem~\ref{stability result} and Lemma~\ref{lower bound}.
    Since $L^2$-equicontinuity is equivalent to tightness on the Fourier side, the equicontinuity of $U$ follows from the exponential decay of $\widehat{Q}$.
\end{proof}
    
With the initial data established, we now prove the main result.
\begin{proof}[Proof of Theorem~\ref{turbulent behavior}]
    Fix initial data $u^\epsilon_0 \in U$ from Lemma~\ref{initial data} for $0 < \epsilon \leq 1$ to be chosen later. Let $u^\epsilon(t)$ denote the maximal solution to \eqref{CM-DNLS} with initial data $u^\epsilon_0$.

    Choose $\lambda(t) > 0$ and $y(t), \theta(t) \in \R$ for all times of existence $t$ so that
        \[ \|u^\epsilon_{\lambda(t), \theta(t), y(t)}(t) - Q\|_{H^1} < \epsilon. \]
    Define 
        \[ q(t, x) = e^{-i\theta(t)}\lambda(t)^{-\frac{1}{2}}Q\Big(\frac{x - y(t)}{\lambda(t)}\Big).\]
    By scaling criticality, we note that $\|u^\epsilon(t) - q(t)\|_{L^2} < \epsilon$.
    We will start by showing that $\lambda(t) \to 0$ and thus $\|q(t)\|_{\dot{H}^s} \to \infty$ for all $s > 0$.
    
    We initially show a uniform upper bound for $\lambda(t)$. Since $\|u^\epsilon(t)\|_{\dot{H}^1} \gtrsim 1$, we find
    \begin{align*}
        1 & \lesssim \|u^\epsilon(t) - q(t)\|_{\dot{H}^1} + \|q(t)\|_{\dot{H}^1} \\
        & = \lambda(t)^{-1}\|u^\epsilon_{\lambda(t),\theta(t),y(t)}(t) - Q\|_{\dot{H}^1} + \lambda(t)^{-1}\|Q\|_{\dot{H}^1} \\
        & \lesssim \lambda(t)^{-1}\epsilon + \lambda(t)^{-1}\|Q\|_{\dot{H}^1}.
    \end{align*}
    Therefore $\lambda(t) \lesssim \epsilon + \|Q\|_{\dot{H}^1} \lesssim 1$ for $0 < \epsilon \leq 1$. 
    
    To conclude $\lambda(t) \to 0$, we use the dispersive decay bounds. By Theorem~\ref{dispersive decay} and inequality \eqref{continuity of Hardy extension}, for all $\Im z > 0$,
    \begin{align*}
        |q(t,z)| & \leq |q(t,z) - u^\epsilon(t,z)| + |u^\epsilon(t,z)| \\
        & \lesssim \|q(t) - u^\epsilon(t)\|_{L^2} (\Im z)^{-\frac{1}{2}} + |t|^{-\frac{1}{2}}\|u^\epsilon_0\|_{L^1}\big[1 + \|u^\epsilon_0\|_{L^2}^2 (1 + (\Im z)^{-1})\big] \\
        & \lesssim \epsilon (\Im z)^{-\frac{1}{2}} + |t|^{-\frac{1}{2}}\|u^\epsilon_0\|_{L^1}\big[ 1 + \|u^\epsilon_0\|_{L^2}^2 (1 + (\Im z)^{-1})\big].
    \end{align*}
    Note that the implicit constants are uniform in $\epsilon$ by the $L^2$-equicontinuity of $U$.
    
    We now evaluate this inequality at the point $z(t) = y(t) + i\lambda(t)$.
    Direct calculation shows
    \begin{equation}\label{lambda bounds 0}
        (2\lambda(t))^{-\frac{1}{2}} = |q(z(t))| \lesssim \epsilon \lambda(t)^{-\frac{1}{2}} + |t|^{-\frac{1}{2}}\|u^\epsilon_0\|_{L^1}\big[1 + \|u^\epsilon_0\|_{L^2}^2(1 + \lambda(t)^{-1})\big].
    \end{equation}
    Because the implicit constants are uniform in $0 < \epsilon \leq 1$, we may take $\epsilon > 0$ sufficiently small to obtain
    \begin{equation}\label{lambda bounds 1}
        \lambda(t)^\frac{1}{2} \lesssim |t|^{-\frac{1}{2}} \|u^\epsilon_0\|_{L^1}\big[1 + \|u^\epsilon_0\|_{L^2}^2(1 + \lambda(t)^{-1})\big],
    \end{equation}
    and so
    \begin{equation*}
        |t|^{\frac{1}{2}} \lesssim \|u^\epsilon_0\|_{L^1}\Big[\big(1 + \|u^\epsilon_0\|_{L^2}^2\big)\lambda(t)^{\frac{1}{2}} + \|u^\epsilon_0\|^2_{L^2}\lambda(t)^{-\frac{1}{2}}\Big].
    \end{equation*}
    The upper bounds on $\lambda(t)$ and $\|u^\epsilon_0\|_{L^2}$ then imply
    \begin{equation*}\label{lambda bounds 2}
        |t|^{\frac{1}{2}} \lesssim \|u^\epsilon_0\|_{L^1}\big[ 1 + \lambda(t)^{-\frac{1}{2}}\big].
    \end{equation*}
    Since $\|u_0^\epsilon\|_{L^1} < \infty$, it follows that $\lambda(t)\lesssim |t|^{-1}$ for sufficiently large $|t|$.

    Using the formula \eqref{Q hat} for $\widehat{Q}$, direct calculation yields
    \begin{align*}
        \|u^\epsilon(t)\|_{\dot{H}^s}^2 & \gtrsim |\lambda(t)|^{-2s}\int_{1/\lambda(t)}^\infty  |\widehat{u^\epsilon}(t,\xi)|^2 d\xi \\
        & \gtrsim |\lambda(t)|^{-2s}\bigg(\int_{1/\lambda(t)}^\infty |\widehat{q}(t, \xi)|^2 d\xi - \epsilon\bigg) \\
        & = |\lambda(t)|^{-2s}\bigg(\int_{1/\lambda(t)}^\infty 4\pi \lambda(t) e^{-2\lambda(t) \xi}d\xi - \epsilon\bigg) \\
        & = |\lambda(t)|^{-2s}\big(2\pi e^{-2} - \epsilon\big).
    \end{align*}
    Choosing $\epsilon > 0$ sufficiently small, this implies that $\|u^\epsilon(t)\|_{\dot{H}^s} \gtrsim |t|^s$ for sufficiently large $|t|$. 
    Because $u^\epsilon$ is either global in time or exhibits finite-time blowup, this concludes the proof of the theorem.
\end{proof}

A natural further question is whether there is interesting behaviour below the $2\pi$ mass threshold.
From the a priori $H^s$ bounds, we cannot expect Sobolev norm growth, but there is potential for some decay in the homogeneous Sobolev spaces $\dot{H}^s$ for $s > 0$.
This is formalized in the following corollary.

\begin{corollary}\label{sobolev decay}
    For all $c, \epsilon, s > 0$ there exists initial data $u_0 \in H^\infty_+$ such that $M(u_0) < 2\pi$ and $\|u_0 - Q\|_{L^2} < \epsilon$ for which the global solution $u(t)$ satisfies
    \begin{equation*}
        \inf_{t \in \R} \|u(t)\|_{\dot{H}^s} \leq c.
    \end{equation*}
\end{corollary}
\begin{proof}
    We begin with the case $s = 1$.
    Suppose for the sake of contradiction that there exist $c_0, \epsilon_0 > 0$ such that for all initial data $u_0 \in H^\infty_+$ with $M(u_0) < 2\pi$ and $\|u_0 - Q\|_{L^2} < \epsilon_0$, the corresponding global solution $u(t)$ satisfies
    \begin{equation*}
        \inf_t \|u(t)\|_{\dot{H}^1} > c_0.
    \end{equation*}
    
    Starting with initial data of the form $(1- \epsilon) Q$ for $0 < \epsilon < \epsilon_0$, we follow the proof of Lemma~\ref{initial data} to construct an $L^2$-equicontinuous family $U = \{u_0^\epsilon : 0 < \epsilon < \epsilon_0\}$ satisfying all the conditions in Lemma~\ref{initial data} along with $M(u_0^\epsilon) < 2\pi$ for all $0 < \epsilon < \epsilon_0$.
    By only smoothing $(1-\epsilon)\widehat{Q}$ near $\xi = 0$, we further ensure for all $r > 0$,
    \begin{equation}\label{convergence}
        \lim_{\epsilon \to 0} \|u_0^\epsilon - Q\|_{H^r} = 0.
    \end{equation}
    This will be necessary to establish the case when $0 \leq s < 1$.
    Note that in the proof of Lemma~\ref{initial data}, we used initial data with a negative eigenvalue in order to gain a lower bound in $\dot{H}^1$.
    In contrast, here we have supposed for the sake of contradiction that such a lower bound exists.
    This is necessary because \eqref{eigenvalue bound} implies that eigenvalues cannot exist below the mass of $2\pi$.
    
    Fix $u_0^\epsilon \in U$ for $0 < \epsilon < \epsilon_0$. Because $M(u_0^\epsilon) < 2\pi$, there exists a global solution $u^\epsilon(t)$ to \eqref{CM-DNLS} with initial data $u_0^\epsilon$.
    Following the same computations as in the proof of Theorem~\ref{turbulent behavior}, we choose sufficiently small $\epsilon > 0$ and find that $\|u^\epsilon(t)\|_{\dot{H}^1} \gtrsim |t|$ for sufficiently large $|t|$. This contradicts the a priori $H^1$ bounds from \cite{gerard-lenzmann}, concluding the proof of the corollary in the case $s = 1$.
    
    For $0 < s < 1$, we use the bound
        \[\|u^\epsilon\|_{\dot{H}^s} \lesssim \|u^\epsilon\|_{\dot{H}^1}^s \|u^\epsilon\|_{L^2}^{1-s}\]
    along with conservation of mass to prove the result.

    For $s > 1$, we choose an integer $n > 2s-1$.
    Since $u_0^\epsilon \to Q$ in $H^n$ by \eqref{convergence}, it follows that the conserved quantities $I_0(u^\epsilon), ..., I_{2n}(u^\epsilon)$ are uniformly bounded and thus $\|u^\epsilon(t)\|_{H^n}$ satisfy global-in-time $H^s$ bounds uniformly in both $t$ and $\epsilon$; see \cite{gerard-lenzmann}.
    We then use the inequality
        \[\|u^\epsilon\|_{\dot{H}^s} \lesssim \|u^\epsilon\|_{\dot{H}^1}^\frac{1}{2} \|u^\epsilon\|_{\dot{H}^{2s-1}}^\frac{1}{2},\]
    to conclude the proof.
\end{proof}

Despite this corollary, it is impossible for individual $H^\infty_+$ solutions to decay to zero in $\dot{H}^1$. 
A quick calculation shows that 
\begin{equation*}
    E(u) \lesssim \big[1 + M(u)^2\big] \|u(t)\|^2_{\dot{H}^1}
\end{equation*}
and so energy provides a non-trivial lower bound for $\|u(t)\|_{\dot{H}^1}$ as soon as $E(u) > 0$.
In \cite{gerard-lenzmann}, it was shown that $\mathcal{L}_uu = 0 $, and hence $E(u) = 0$, if and only if either $u = 0$ or $u = Q_{\lambda, \theta, y}$ for some $\lambda > 0$ and $\theta, y \in \R$.
Since these solutions do not evolve in time, this rules out the possibility of an individual nonzero solution satisfying $\inf_t\|u(t)\|_{\dot{H}^1} = 0$.

For $s > 1$, the bound 
$$\|f\|_{\dot{H}^1} \lesssim \|f\|^{1-\frac{1}{s}}_{L^2}\|f\|^{\frac{1}{s}}_{\dot{H}^s}$$
shows that decay in $\dot{H}^s$ implies decay in $\dot{H}^1$.
Thus there cannot exist an individual solution $u$ in $H^\infty_+$ which satisfies $\inf_t\|u(t)\|_{\dot{H}^s} = 0$.
For $0 \leq s < 1$, the global-in-time bounds on $H^s$ and the bound
$$\|f\|_{\dot{H}^1} \lesssim \|f\|^{\frac{1}{2}}_{\dot{H}^s}\|f\|^{\frac{1}{2}}_{\dot{H}^{2-s}}$$
imply that the same statement holds for $H^\infty_+$ solutions with mass below $2\pi$.

\bibliographystyle{plain}
\bibliography{references}
\end{document}